\documentclass[a4paper]{amsart}

\input{article.sty}

\title{The \texorpdfstring{$L^\infty$}{L infinity}-positivity preserving property and stochastic completeness}
\date{\today}

\author{Andrea Bisterzo}
\address[A. Bisterzo]{Dipartimento di Matematica e Applicazioni,
Università degli Studi di Milano\hyp{}Bicocca, Via R. Cozzi 55, I-20125, Milano}
\email{a.bisterzo@campus.unimib.it
}
\author{Ludovico Marini} 
\address[L. Marini]{Dipartimento di Matematica e Applicazioni,
Università degli Studi di Milano\hyp{}Bicocca, Via R. Cozzi 55, I-20125, Milano}
\email{l.marini9@campus.unimib.it}

\begin{document}

\begin{abstract}
    We say that a Riemannian manifold satisfies the $L^p$-positivity preserving property if  $(-\Delta + 1)u\ge 0$ in a distributional sense implies $u \ge 0$ for all $ u \in L^p$.
    While geodesic completeness of the manifold at hand ensures the $L^p$-positivity preserving property for all $p \in (1, +\infty)$, when $p = + \infty$ some assumptions are needed. 
    In this paper we show that the $L^\infty$-positivity preserving property is in fact equivalent to stochastic completeness, i.e., the fact that the minimal heat kernel of the manifold preserves probability. 
    The result is achieved via some monotone approximation results for distributional solutions of $-\Delta + 1 \ge 0$, which are of independent interest. 
\end{abstract}

\maketitle

\section{Introduction} 
\label{sec:introduction}
Let $(M, g)$ be an $n$-dimensional Riemannian manifold with Riemannian measure $d\mu_g$. 
In the following $\Delta = \Div\nabla$ is the (negatively defined) Laplace-Beltrami operator and, unless explicitly stated, all integrals are taken with respect to the Riemannian volume measure $d\mu_g$.

The aim of this paper is to study qualitative properties for certain solutions of elliptic PDEs involving the following Schr\"odinger operator
\begin{equation}
    \label{eq:schrodinger operator L}
    \mathcal{L}  \coloneqq \Delta  - 1. 
\end{equation}
We say that $ u \in L^1_{\loc}(M)$ solves $(-\Delta  + 1)u \geq 0$ \textit{in the sense of distributions} if 
\begin{equation*}
   \int_M u(-\Delta + 1)\varphi \geq 0
\end{equation*}
for all $\varphi \in C^\infty_c(M)$ with $\varphi \ge 0$.
Note that this is equivalent to say that $(-\Delta + 1)u$ is a positive Radon measure. 
If more regularity is assumed, namely $u \in W^{1, 2}_{\loc}(M)$, we talk of a \textit{weak} solution of $(-\Delta + 1) u \ge 0$ if
\begin{equation*}
   \int_M g(\nabla \varphi, \nabla u) + u\varphi \geq 0
\end{equation*}
for all $\varphi \in C^\infty_c(M)$ with $\varphi \ge 0$.
Finally, if $u \in C^2(M)$, then $(-\Delta + 1) u \geq 0$ is intended in a strong, pointwise sense. 
Naturally, if $u \in C^2(M)$ is a strong solution of the inequality, it is also a weak and thus distributional solution.

We begin with the following definition:
\begin{defin}
\label{def:L^p pp}
Let $p \in [1, + \infty]$, we say that $(M,g)$ has the  \textit{$L^p$-positivity preserving property} if every $u\in L^p(M)$ satisfying 
\begin{equation}
    \label{eq: -delta + 1 >= 0}
    (-\Delta + 1) u \ge 0
\end{equation}
in the sense of distributions is non-negative a.e.
\end{defin}
The definition was proposed by G\"uneysu in \cite{G2016} although the case $p=2$ of this property appears in previous works of Kato, \cite{K1972}, and Braverman, Milatovic and Shubin, \cite{BMS02}.
In this last paper, the authors proved that the validity of the $L^2$-positivity preserving property implies the essential self-adjointness of the Schr\"odinger operator $-\Delta + V$ for all $L^2_{\loc}$ non-negative potentials $V$. 
Recall that $-\Delta+V:C^\infty_c(M)\to L^2(M)$ is \textit{essentially self-adjoint} if it has an unique self-adjoint extension to $L^2(M)$ (its closure).
On the other hand, the operator $-\Delta + V$ is known to be essentially self-adjoint on geodesically complete manifolds, see \cite[Corollary 2.9]{shubin2001essential} or \cite{BMS02} and \cite{guneysu2013path} for the case of operators acting on Hermitian vector bundles. 
The combination of these results lead Braverman, Milatovic and Schubin to formulate the following

\begin{conjecture}[BMS]
If $(M, g)$ is a geodesically complete Riemannian manifold then the $L^2$-positivity preserving property holds.
\end{conjecture}

This conjecture has remained open for 20 years and has only recently been solved in the positive by Pigola and Veronelli in \cite{PV2021}.
For a complete introduction to the topic we refer to the nice survey \cite{Gu17BMS}, to Chapter XIV.5 of \cite{G2017} and Appendix B of \cite{BMS02}.

The case $p = + \infty$ of Definition \ref{def:L^p pp} is instead related to stochastic completeness. 
Recall that a manifold is said to be stochastically complete if the Brownian paths on $M$ have almost surely infinite lifetime or, equivalently, if the minimal positive heat kernel associated to the Laplace-Beltrami operator preserves probability.
For the scope of this article, however, we shall adopt the following (equivalent) definition, which is more relevant from the point of view of PDEs. 

\begin{defin}
\label{def:stoch compl}
A Riemannian manifold $(M,g)$ is said to be \textit{stochastically complete} if the only bounded, non-negative $C^2$ solution of $\Delta u\geq\ u$ on $M$ is $u \equiv 0$.
\end{defin}

There are countless characterizations of stochastic completeness, a comprehensive account is beyond the scope of this paper, we refer the reader to \cite{Gr1999, Gr2009, PRS2003, PRS2005} or the very recent \cite{GIM2020}. 
See also \Cref{sec: L infty pp and stoch compl} below.
Stochastic completeness is implied by several geometric, analytic and probabilistic conditions. For instance, stochastic completeness is ensured by conditions on the curvature tensor. 
In this direction, the most general result is the one of Hsu in \cite{Hsu1989}, a particular case of which states that geodesically complete manifold whose Ricci curvature satisfies  
\begin{equation*}
    \Ric(x) \ge -C r^2(x)
\end{equation*}
outside a compact set are in fact stochastically complete. 

As a matter of fact, the $L^p$-positivity preserving property implies stochastic completeness of the manifold at hand, as it has been observed by G\"uneysu in \cite{G2016}. 
In particular, stochastically incomplete manifolds provide counterexamples to the validity of the $L^\infty$-positivity preserving property. 
As an example, take a Cartan-Hadamard manifold whose Ricci curvature diverges at $-\infty$ faster than quadratically, for computations we refer to \cite{MV2021}.

In the last years there has been an effort to better understand the $L^p$-positivity preserving property and to find geometric and analytic conditions ensuring its validity.
If one takes $M = \R^n$ with the usual Euclidean metric, the $L^2$-positivity preserving property was first proved by Kato, \cite{K1972}, using the theory of operators on tempered distributions. 
In a Riemannian setting, however, one does not dispose of tempered distributions and it is thus necessary to take other paths. 
Following an idea of Davies in \cite{BMS02}, if the manifold admits a family of smooth cutoff functions with a good control on the Laplacian, it is possible to prove the $L^p$-positivity preserving property. 
In this direction we mention the results by Braverman, Milatovic and Schubin in \cite{BMS02}; by G\"uneysu in \cite{G2016, G2017}; by Bianchi and Setti in \cite{BS2018}; and by the second author and Veronelli in \cite{MV2021}. 
Using a completely different strategy, Pigola and Veronelli, \cite{PV2021}, were finally able to prove that the $L^p$-positivity preserving property for $p \in (1, + \infty)$ holds on geodesically complete manifolds, thus verifying that the BMS conjecture is true. 
\begin{remark}
\label{rmk:counterexample incomplete mfd}
Without geodesic completeness the $L^p$-positivity preserving property generally fails for every $p \in [1, +\infty]$. 
To see this, take $\mathbb{B}_1\subseteq \mathbb{R}^2$ the Euclidean open ball of radius $1$. 
Then, the radial function $u(r) = -r$, which belongs to all $L^p$ spaces, is a non-positive function which satisfies $(-\Delta + 1) u \le 0$. 
\end{remark}
The proof of Pigola and Veronelli uses some new regularity results for non-negative subharmonic distributions to prove that the $L^p$-positivity preserving property is implied by a Liouville-type property for $L^p$-subharmonic distributions. When $p\neq 1, +\infty$, this property is known to hold on geodesically complete manifolds thanks to a result of Yau, \cite{Y1976}. 
This strategy, however, fails when $p = 1$ or $p = +\infty$ since there are known counterexamples to the Liouville-type property.

\begin{remark}
\label{rmk:observation about result of Hsu}
To the best of our knowledge, when $p = +\infty$ the most general condition known so far ensuring the validity of the $L^\infty$-positivity preserving property is the one of Theorem II in \cite{MV2021}. 
This condition, which requires geodesic completeness and 
\begin{equation*}
    \Ric(x) \ge -Cr^2(x)
\end{equation*}
outside a compact set, is essentially the celebrated condition of Hsu, \cite{Hsu1989}, for stochastic completeness.
\end{remark}
The above observation suggests a much closer relation between stochastic completeness and the $L^\infty$-positivity preserving property.
The main result of this article is in fact the following:

\begin{mytheorem}
\label{thm: main theorem introduction}
Let $(M, g)$ be a Riemannian manifold, then $M$ has the $L^\infty$-positivity preserving property if and only if it is stochastically complete. 
\end{mytheorem}

\Cref{thm: main theorem introduction} together with the result of Pigola and Veronelli, \cite{PV2021}, give the full picture of the $L^p$-positivity preserving property when $p \in (1, + \infty]$.  
When $p = 1$, the best result we have is the one of the second author with Veronelli, \cite[Theorem II]{MV2021}, which ensures the $L^1$-positivity preserving property if the manifold is complete and the Ricci curvature essentially grows like 
\begin{equation*}
    \Ric(x) \ge -Cr^2(x)
\end{equation*}
outside of a compact set. 
Using a construction suggested to us by Veronelli we also prove the following:

\begin{mytheorem}
\label{thm:L1 ppp intro}
For every $\varepsilon > 0$, there exists a 2-dimensional Riemannian manifold $(M, g)$ whose Gaussian curvature satisfies 
\begin{equation*}
    K(x) \sim -C r(x)^{2+\varepsilon},
\end{equation*}
such that the $L^1$-positivity preserving property fails on $M$. 
Here $r(x)$ denotes the Riemannian distance from some fixed pole. 
\end{mytheorem}
\begin{remark}
\Cref{thm:L1 ppp intro} together with \Cref{rmk:counterexample incomplete mfd} show that the result of Theorem II in \cite{MV2021} alluded in the above is optimal. 
\end{remark}
\begin{remark}
Using a simple trick introduced in \cite{HMRV2021}, the counterexample in dimension $2$ of \Cref{thm:L1 ppp intro} can be used to construct counterexamples to the $L^1$-positivity preserving property in arbitrary dimensions $n \geq 2$. 
It suffices to take the product of the 2 dimensional model manifold $M$ with an arbitrary $n-2$ dimensional closed Riemannian manifold. 
Extending the function which provides the counterexample on $M$ to the whole product produces a counterexample in a manifold of dimension $n$. 
\end{remark}

In order to prove that stochastic completeness implies the $L^p$-positivity preserving property, we show that it is essentially a problem of regularity for the distributional, $L^\infty$ solutions of $\mathcal{L} u \geq 0$, where $\mathcal{L}$ is defined in \eqref{eq:schrodinger operator L}. 
Using a Brezis-Kato inequality we reduce ourselves to prove the following:
\begin{myproposition}
\label{prop: Claim in proof of thm A}
Let $(M, g)$ be a Riemannian manifold and let $u \in L^\infty(M)$ satisfying $\mathcal{L} u \geq 0$ in the sense of distributions.
Then, there exists some $w \in C^{\infty}(M)$ with $\sup_M w < + \infty$ such that $u \leq w$ and $\mathcal{L} w \geq 0$ in the strong sense. 
\end{myproposition}

This latter result follows from a monotone approximation theorem for the distributional solutions of $\mathcal{L} u \ge 0$ which is of independent interest. 

\begin{mytheorem}
\label{thm:monotone approximation L intro}
Let $(M, g)$ be a Riemannian manifold and let $u \in L^1_{\loc}(M)$ be a solution of $\mathcal{L}u \ge 0$ in the sense of distributions. Then for every $\Omega \Subset M$ there exists a sequence $\lbrace u_k \rbrace \subset C^\infty(\Omega)$ such that:
\begin{enumerate}[label = (\roman*)]
    \item $u_k \searrow u$ pointwise a.e.;
    \item $\mathcal{L}u_k \ge 0$ for all $k$;
    \item $u_k \to u$ in $L^1(\Omega)$;
    \item $\sup_\Omega u_k \leq 2 \esup_\Omega u$.
\end{enumerate}
\end{mytheorem}

Using a trick due to Protter and Weinberger, \cite{PW1984}, it is sufficient to prove a monotone approximation result for the distributional solution of $\Delta_\alpha v \geq 0$, where $\Delta_\alpha v\coloneqq \alpha^{-2} \Div(\alpha^2 \nabla v)$ and $\alpha$ is a smooth positive function to be specified later. 
The monotone approximation for the weighted Laplacian is obtained using a strategy outlined by Bonfiglioli and Lanconelli in \cite{bonfiglioli2013subharmonic} together with some mean value representation formulas for the solution of $\Delta_\alpha v = 0$. 
\Cref{thm:monotone approximation L intro} generalizes a result of Pigola and Veronelli in \cite{PV2021} where the monotone approximation was proved only on coordinate charts. 
If the manifold at hand admits a minimal, positive Green function for the operator $\Delta_\alpha$ (i.e. it is $\alpha$-non-parabolic) and if this Green function vanishes at infinity (i.e. it is strongly $\alpha$-non-parabolic), as a byproduct of the proof of \Cref{thm:monotone approximation L intro} we obtain a global, monotone approximation result. 

\begin{mycorollary}
\label{Cor:GlobalNonParabolic}
Let $(M, g)$ be a strongly $\alpha$-non-parabolic Riemannian manifold and let $u \in L^1_{\loc}(M)$ be a solution of $\mathcal{L}u \ge 0$ in the sense of distributions. Then there exists a sequence $\lbrace u_k \rbrace \subset C^\infty(M)$ such that:
\begin{enumerate}[label = (\roman*)]
    \item $u_k \searrow u$ pointwise a.e.;
    \item $\mathcal{L}u_k \ge 0$ for all $k$;
    \item $u_k \to u$ in $L^1(M)$;
    \item $\sup_M u_k \leq 2 \esup_M u$.
\end{enumerate}
\end{mycorollary}
\begin{remark}
Results such as \Cref{prop: Claim in proof of thm A}, \Cref{thm:monotone approximation L intro} and \Cref{Cor:GlobalNonParabolic} still hold if the constant $1$ in the operator \eqref{eq:schrodinger operator L} is replaced by another positive constant. 
Actually, negative constants are also allowed as long as $-\mathcal{L}$ remains a positive operator. 
\end{remark}

The paper is organized as follows. 
In \Cref{sec: L infty pp and stoch compl} we study the relation between the $L^\infty$-positivity preserving property and stochastic completeness, showing that the former property implies the latter and the converse is true up to a claim which is proved later on. 
\Cref{sec:monotone approximation} is devoted to the monotone approximation results.
We first observe that the conclusions of \Cref{thm:monotone approximation L intro} can be inferred form an equivalent statement for the operator $\Delta_\alpha$. 
We prove some mean value representation formulae for $\alpha$-harmonic functions and show how these can be used to produce a monotone approximating sequence with wanted properties. 
As a corollary of Theorem \ref{thm:monotone approximation L intro}, we obtain the desired claim which concludes the proof of Theorem \ref{thm: main theorem introduction}. 
We end the section by observing that if we make some assumptions on the geometry of $M$, the monotone approximation results have a global nature. 
Finally, in \Cref{sec:counterexample L1} we construct a class of Riemannian manifolds on which the $L^1$-positivity preserving property fails thus proving \Cref{thm:L1 ppp intro}. 

\section{\texorpdfstring{$L^\infty$}{L infinity}-positivity preserving property and stochastic completeness}
\label{sec: L infty pp and stoch compl}

The aim of this section is to investigate the connection between the $L^\infty$-positivity preserving property and stochastic completeness. 
As pointed out in the introduction, there are several possible definitions one can give for stochastic completeness. 
We cite here the ones relevant to our exposition. 
\begin{enumerate}[label = (\roman*)]
\item for every $\lambda > 0$, the only bounded, non-negative $C^2$ solution of $\Delta u \geq \lambda u$ is $u \equiv 0$;
\item for every $\lambda > 0$, the only bounded, non-negative $C^2$ solution of $\Delta u = \lambda u$ is $u \equiv 0$;
\item the only bounded, non-negative $C^2$ solution of $\Delta u = u$ is $u \equiv 0$.
\end{enumerate}
For a proof of the equivalence we refer to Theorem 6.2 in \cite{Gr1999}. 
\begin{remark}
\label{rmk:observation on regularity of stoch completeness}
Note that the regularity required in the above and in \Cref{def:stoch compl} can be relaxed to $C^0(M) \cap W^{1, 2}_{\loc}(M)$ see for instance Section 2 of \cite{AMR2016}.
This fact is a consequence of a stronger version of \Cref{thm:sub super solution} below. 
\end{remark}

We begin with the following observation due to G\"uneysu, \cite{G2016}.

\begin{prop}
\label{prop: L infty implies stoch compl}
If $(M,g)$ has the $L^\infty$-positivity preserving property, then it is stochastically complete.
\begin{proof}
To see this, take $u \in C^2(M)$ a bounded and non-negative function satisfying $\Delta u \ge u$. 
Then, if we set $v = -u$ we have 
\begin{equation*}
    v \in L^\infty(M) \qquad (-\Delta + 1) v \ge 0.
\end{equation*}
By the $L^\infty$-positivity preserving property, we conclude that $v \geq 0$, since $u$ is non-negative, this yields $v \equiv 0$ and hence $u \equiv 0$. 
\end{proof}
\end{prop}

\begin{remark}
\label{rmk: stochastical vs geodesic completeness}
It is worthwhile noticing that stochastic completeness is in general unrelated to geodesic completeness. It is possible to find Riemannian manifolds which are geodesically but not stochastically complete such as Cartan-Hadamard manifolds whose Ricci curvature diverges at $-\infty$ faster that quadratically. 
On the other hand, $\R^n \setminus \lbrace 0 \rbrace$ endowed with the Euclidean metric is stochastically complete but geodesically incomplete. 
\end{remark}

Proposition \ref{prop: L infty implies stoch compl} and the above remark explain the failure of the result of Pigola and Veronelli, \cite{PV2021}, in the case $p = \infty$. 

\subsection{From stochastic completeness to the \texorpdfstring{$L^\infty$}{L infinity}-positivity preserving property}
The goal of this section is to set the ground towards proving the converse of \Cref{prop: L infty implies stoch compl}. 

To this end, let $(M, g)$ be a stochastically complete Riemannian manifold and take $u \in L^\infty(M)$ satisfying $(-\Delta + 1) u \ge 0$ in the sense of distributions. 
Our purpose is to show that $u$ is non-negative almost everywhere or, equivalently, that the negative part $u_- = \max\lbrace 0, -u\rbrace = (-u)_+$ vanishes a.e.. 
The next ingredient in our proof is the following Brezis-Kato inequality due to Pigola and Veronelli, \cite[Proposition 4.1]{PV2021}

\begin{thm}[Brezis-Kato]
\label{Thm:Kato}
Given a Riemannian manifold $(M,g)$, if $u\in L^1_{\loc}(M)$ satisfies $\mathcal{L}u \ge 0$ in the sense of distributions, then $u_+\in L^1_{\loc}(M)$ and $\mathcal{L}u_+ \ge 0$ in the sense of distributions.
\end{thm}

Since $\mathcal{L}(-u) \ge 0$ we conclude that $\mathcal{L}u_- \ge 0$ in the sense of distributions. If $u_-$ happened to be a $C^2(M)$ function, stochastic completeness would allow us to conclude that $u_- \equiv 0$, hence $u \ge 0$. 
Note that, according to \Cref{rmk:observation on regularity of stoch completeness}, $u_- \in C^0(M) \cap W^{1, 2}_{\loc}(M)$ would be sufficient. 
In general, however, this is not the case and, as a matter of fact, it is a stronger requirement than what we actually need. 
Indeed, if we find $w \in C^2(M)$ such that $\sup_M w < + \infty$, $0 \le u_- \le w$ and $\mathcal{L}w \ge 0$, then stochastic completeness applied to $w$ implies that $w$ hence $u_-$ are identically zero. 

The existence of such function $w$ is implied by the following corollary of \Cref{thm:monotone approximation L intro}, whose proof is postponed to the next section.

\begin{corollary}
\label{cor:CLAIM 2}
Let $(M, g)$ be a Riemannian manifold and let $u \in L^\infty(M)$ be a distributional solution of $\mathcal{L}u \ge 0$. 
Then, for every relatively compact $\Omega \Subset M$ there exists some $u_\Omega \in C^\infty(\Omega)$ which solves $\mathcal{L}u_{\Omega} \ge 0$ in a strong sense and such that $u \le u_\Omega \le 2\esup_\Omega u$.
\end{corollary}

Via a compactness argument we use the functions $u_\Omega$ to construct the function $w$.
The following theorem, proved by Sattinger in \cite{sattinger1972monotone}, also comes into aid as it allows to obtain $\mathcal{L}$-harmonic function from super/sub solutions of $\mathcal{L} u = 0$. 

\begin{thm}
\label{thm:sub super solution}
Let $u_1, u_2 \in C^\infty(M)$ satisfying 
\begin{equation*}
    \mathcal{L} u_1 \ge 0, \quad  \mathcal{L} u_2 \le 0, \quad  u_1 \le u_2
\end{equation*}
on $M$. Then, there exists some $w \in C^\infty(M)$ such that
\begin{equation*}
    u_1 \le w \le u_2 \quad \text{ and } \quad \mathcal{L} w = 0.
\end{equation*}
\end{thm}

\begin{remark}
\Cref{thm:sub super solution} is a weaker formulation of a much more general theorem, proved by Ratto, Rigoli and Véron, \cite{RRV1994}, for a wider class of functions, namely $u_1,u_2\in C^0(M)\cap W_\textnormal{loc}^{1,2}(M)$. 
This result goes under the name of sub and supersolution method or monotone iteration scheme. Note that the results of \cite{RRV1994} hold for a larger class of second order elliptic operators. 
For a survey on the subject, we refer to Heikkil\"a and Lakshmikantham, \cite{heikkila2017monotone}.
\end{remark}

Using the functions constructed locally in \Cref{cor:CLAIM 2} together with an exhaustion procedure we obtain the following:

\begin{thm}
\label{thm:CLAIM 1}
Let $(M,g)$ be a Riemannian manifold and let $u\in L^\infty(M)$ satisfying $\mathcal{L}u\geq 0$ in the sense of distributions.
Then, there exists $w \in C^\infty(M)$ such that $u\leq w$,  $\mathcal{L}w \geq 0$ in a strong sense and $\sup_M w < + \infty$.
\begin{proof}
We begin by observing that if $u \in L^\infty(M)$ then, setting $c = \Vert u \Vert_{L^\infty(M)}$, we have
\begin{equation*}
    \mathcal{L} c = - c \le 0 \text{ on }M. 
\end{equation*}
Next, take $\lbrace \Omega_h \rbrace$ an exhaustion of $M$ by relatively compact sets such that
\begin{equation*}
    \Omega_1 \Subset \Omega_2 \Subset \ldots \Subset \Omega_h \Subset \Omega_{h+1} \Subset \ldots \Subset M, 
\end{equation*}
$\partial \Omega_h$ is smooth and $M=\cup_h \Omega_h$. 
On each set $\Omega_h$ we apply \Cref{cor:CLAIM 2} and we obtain a sequence of functions $u_h \in C^\infty(\Omega_h)$ such that
\begin{enumerate}
\item $u\leq u_h \leq 2c$ in $\Omega_h$;
\item $\mathcal{L}u_h\geq 0$ strongly on $\Omega_h$.
\end{enumerate}
Since $\mathcal{L}c\leq 0$, we use \Cref{thm:sub super solution} on each $\Omega_h$ to obtain $w_h\in C^\infty(\Omega_h)$ satisfying
\begin{enumerate}
\item $\mathcal{L}w_h=0$;
\item $u_h\leq w_h \leq 2c$.
\end{enumerate}
We conclude by showing that $\{w_h\}_h$ is bounded respect to the $C^\infty(M)$-topology and thus converges, up to a subsequence, to some $w \in C^\infty(M)$.

To this end, let $K \subset M$ be a compact set and $k \in \N$, $k \geq 2$. 
By Schauder estimates for the operator $\mathcal{L}$ we have
\begin{equation*}
    \Vert w_h \Vert_{C^{k}(K)} \leq A \left(\Vert w_h \Vert_{L^\infty(K)} + \Vert \mathcal{L}w_h \Vert_{C^{k-2, \alpha}(K)}  \right)
\end{equation*}
for some $\alpha \in (0, 1)$. 
See for instance Section 6.1 of \cite{GT2001}.  
In particular there exists a constant $C=C(K,n, k)>0$ such that $\Vert w_h \Vert_{C^k(K)}<C$ for every $h\in \N$. 
Here
\begin{align*}
    ||w_h||_{C^k(K)}=||w_h||_{L^\infty(K)}+||\nabla w_h||_{L^\infty(K)} + \cdots + ||\nabla^k w_h||_{L^\infty(K)}.
\end{align*}
Since $\{w_h\}_h$ is pre-compact, it converges in the $C^\infty(M)$ topology up to a subsequence, denoted again with $\{w_h\}_h$. 
Let $w \in C^\infty(M)$ be the $C^\infty$ limit, we have that
\begin{equation*}
    u\leq w\leq 2c \quad \textnormal{and}\quad \mathcal{L}w=0.
\end{equation*}
\end{proof}
\end{thm}

This concludes the proof of \Cref{thm: main theorem introduction}, apart from the proof of \Cref{cor:CLAIM 2}. 

\section{Monotone approximation results}
\label{sec:monotone approximation}

This section is devoted to the proof of \Cref{thm:monotone approximation L intro}.
Instead of proving \Cref{thm:monotone approximation L intro} directly, 
we prove an equivalent monotone approximation result for another elliptic differential operator closely related to $\mathcal{L}$. 
We begin by taking a function $\alpha\in C^\infty (M)$ satisfying
\begin{equation}
\label{Eq:AlphaProblem}
    \begin{cases}
    \mathcal{L}\alpha=0\\
    \alpha>0
    \end{cases}.
\end{equation}
The existence of such a function is ensured by \cite{fischer1980structure}, and is equivalent to the fact that $\lambda_1^{-\mathcal{L}} >0$.
In our case it is easy to see that $\lambda^{-\mathcal{L}}_1 \geq 1$.

Using $\alpha$ we define the following drifted Laplacian
\begin{equation}
    \label{eq:Delta alpha}
    \Delta_\alpha:u\mapsto \alpha^{-2} \Div(\alpha^2 \nabla u).
\end{equation}
With a trivial density argument, one has that $\Delta_\alpha$ is symmetric in $L^2$ with respect to the measure $\alpha^2 d\mu_g$.
Then, using the following idea due to Protter and Weinberger, \cite{PW1984}, we establish the relation between $\Delta_\alpha$ and $\mathcal{L}$. 
See also Lemma 2.3 of \cite{PV2021}.

\begin{lem}
\label{Lem:EquivalenceOfPositivity}
If $u\in L^1(\Omega)$ with $\Omega \Subset M$, then
\begin{equation*}
(\Delta-1)u\geq 0 \qquad \Leftrightarrow \qquad \Delta_\alpha\left(\frac{u}{\alpha}\right)\geq 0,
\end{equation*}
where both inequalities are intended in the sense of distributions.

\begin{proof}
Fix $0\leq\varphi\in C^\infty_c(\Omega)$, by direct computation we have
\begin{equation}
\label{Lem:EquivalenceOfPositivity_Formula1}
\begin{split}
\alpha \Delta_\alpha \left( \frac{\varphi}{\alpha} \right) 
& =\alpha^{-1} \Div\left[\alpha^2 \nabla \left(\frac{\varphi}{\alpha}\right)\right]\\
&=\alpha^{-1} \Div\left(\alpha \nabla \varphi - \varphi \nabla \alpha\right)\\
&=\Delta \varphi - \varphi \frac{\Delta \alpha}{\alpha}\\
&=\mathcal{L}\varphi,
\end{split}
\end{equation}
where in the last equation we have used \eqref{Eq:AlphaProblem}.
Thus, using \eqref{Lem:EquivalenceOfPositivity_Formula1} and the symmetry of $\D_\alpha$ we conclude 
\begin{equation*}
\begin{split}
\left(\Delta_\alpha \left(\frac{u}{\alpha}\right), \alpha \varphi \right)_{L^2} & =\int_\Omega{\frac{u}{\alpha}\ \Delta_\alpha \left(\frac{\varphi}{\alpha}\right)\ \alpha^2 d\mu_g}\\
&=\int_\Omega{u\ (\Delta-1)\varphi\ d\mu_g}=\left((\Delta-1)u,\varphi\right)_{L^2}.
\end{split}
\end{equation*}
\end{proof}
\end{lem}

Using \Cref{Lem:EquivalenceOfPositivity_Formula1} and setting $v = \alpha^{-1} u$, it is possible to obtain \Cref{thm:monotone approximation L intro} from an equivalent statement for the operator $\Delta_\alpha$. In this perspective, our goal is to prove the following:

\begin{thm}
\label{Thm:MonotoneApproximationDriftedLaplacian}
Let $(M, g)$ be a Riemannian manifold and let $v \in L^1_{\loc}(M)$ be a solution of $\D_\alpha v \ge 0$ in the sense of distributions. Then, for every $\Omega \Subset M$ there exists a sequence $\lbrace v_k \rbrace \subset C^\infty(\Omega)$ such that:
\begin{enumerate}[label = (\roman*)]
    \item $v_k \searrow v$ pointwise a.e.;
    \item $\D_\alpha v_k \ge 0$ for all $k$;
    \item $v_k \to v$ in $L^1(\Omega)$;
    \item $\sup_\Omega v_k \leq \esup_\Omega v$.
\end{enumerate}
\end{thm}


\subsection{Representation formula for \texorpdfstring{$\alpha$}{alpha}-harmonic functions}
We begin by establishing some mean value representation formulae involving the Green function of the operator $\D_\alpha$ on $\Omega$ with Dirichlet boundary conditions.
Recall that $G : \overline{\Omega} \times \overline{\Omega} \setminus \{x = y\} \to \R$ is a symmetric, $L^1(\Omega\times \Omega)$ function satisfying the following properties:

\begin{enumerate}[label =(\alph*)]
    \item $G \in C^\infty\left(\Omega \times \Omega \setminus \{x = y\}\right)$ and $G(x, y) > 0$ for all $x, y \in \Omega$ with $x \neq y$;
    \item $\lim_{x \to y} G(x, y) = + \infty$ and $G(x, y) = 0$ if $x \in \partial \Omega$ (or $y \in \partial \Omega$);
    \item $\D_\alpha G (x, y) = - \delta_x(y)$ with respect to $\alpha^2 d\mu_g$, that is, 
    \begin{equation*}
        \varphi(x) = -\int_\Omega G(x, y) \D_\alpha \varphi(y) \alpha^2(y) d\mu_y \qquad \forall \varphi \in C^\infty_C(\Omega)
    \end{equation*}.
\end{enumerate}

For $r > 0$ and $x \in \Omega$, we define the following set
\begin{equation}
    \label{eq:balls green}
    \mathcal{B}_r(x)\coloneqq\left\lbrace y\in \Omega\ |\ G(x,y)>r^{-1} \right \rbrace \cup \{ x\}. 
\end{equation}
We adopt the convention $G(x, x) = +\infty$ so that  $\mathcal{B}_r(x) = \left\lbrace y\in \Omega\ |\ G(x,y)>r^{-1} \right \rbrace$. 
Observe that $\mathcal{B}_r(x) \subset \Omega$ are open and relatively compact sets, moreover, for almost all $r>0$, $\partial \mathcal{B}_r(x)$ is a smooth hypersurface. This is a consequence of Sard's theorem. 
In the following, $d\sigma$ and $d\mu$ represent the Riemannian surface and volume measure of $\partial\mathcal{B}_r(x)$ and $\mathcal{B}_r(x)$ respectively. 

\begin{prop}
\label{Prop:RepresentationFormula}
For every $v\in C^\infty(\Omega)$ and almost every $r>0$, the following representation formula holds
\begin{equation}
\label{Eq:RepresentationFormula}
\begin{split}
v(x)=\int_{\partial \mathcal{B}_r(x)} v(y) |\nabla G (x,y)| \alpha^2(y) d\sigma_y 
- \int_{\mathcal{B}_r(x)} \left[G(x,y)-\frac{1}{r}\right] \Delta_\alpha v(y) \alpha^2(y) d\mu_y
\end{split}
\end{equation}
\end{prop}
\begin{proof}
By the Green identity we have 
\begin{align*}
    v(x) = &-\int_{\mathcal{B}_r(x)} G(x,y) \D_\alpha v(y) \alpha^2(y) d\mu_y \\
    &+ \int_{\partial \mathcal{B}_r(x)} \Big(G (x,y) \frac{\partial v}{\partial \nu}(y) -v(y) \frac{\partial G}{\partial \nu}(x,y)\Big)\alpha^2(y) d\sigma_y.
\end{align*}
Since $\frac{\partial G}{\partial \nu}=-|\nabla G|$, we obtain
\begin{align*}
    v(x)=&\int_{\partial \mathcal{B}_r(x)} v(y) \Big|\nabla G (x,y)\Big| \alpha^2(y) d\sigma_y+\frac{1}{r}\int_{\partial \mathcal{B}_r(x)} \frac{\partial v}{\partial \nu}(y) \alpha^2(y) d\sigma_y \\
    &- \int_{\mathcal{B}_r(x)} G(x,y) \Delta_\alpha v(y) \alpha^2(y) d\mu_y\\
    =&\int_{\partial \mathcal{B}_r(x)} v(y) \Big|\nabla G (x,y)\Big| \alpha^2(y) d\sigma_y- \int_{\mathcal{B}_r(x)} \Big[G(x,y)-\frac{1}{r}\Big] \Delta_\alpha v(y) \alpha^2(y) d\mu_y.
\end{align*}
\end{proof}

In particular, if $v\in C^2(\Omega)$ is $\alpha$-harmonic, i.e. $\Delta_\alpha u=0$ on $\Omega$, then
\begin{equation}
\label{Eq:RepresentationFormulaAlphaHarmonic}
v(x)=\int_{\partial \mathcal{B}_r(x)}|\nabla G(x,y)|\ v(y)\ \alpha^2(y)\ d\sigma_y.
\end{equation}
The formulae \eqref{Eq:RepresentationFormulaAlphaHarmonic} and \eqref{Eq:RepresentationFormula} are a generalization of some standard representation formula for the Laplace-Beltrami operator. 
See for instance the Appendix of \cite{bonfiglioli2013subharmonic}, \cite{N2007} or the very recent \cite{cupini2021mean}.

\subsection{Distributional vs. potential \texorpdfstring{$\alpha$}{alpha}-subharmonic solutions}

Before proving the monotone approximation result, we observe that the notion of $\alpha$-subharmonicity in the distributional sense is closely related to the notion of $\alpha$-subharmonic solutions in the sense of potential theory. 

\begin{defin} 
\label{def:alpha subharmonic potential}
We say that an upper semicontinuous function $u:\Omega\to [-\infty,+\infty)$ is \textit{$\alpha$-subharmonic in the sense of potential theory} on $\Omega$ if the following conditions hold
\begin{enumerate}[label = (\roman*)]
\item $\{x\in \Omega\ |\ u(x)>-\infty\}\neq \emptyset$;
\item for all $V\Subset \Omega$ and for every $h\in C^2(V)\cap C^0(\overline{V})$ such that $\D_\alpha h = 0 $ in $V$ with $u \le h$ on $\partial V$, then
\begin{equation*}
u\leq h \qquad \text{ in }  V. 
\end{equation*}
\end{enumerate}
\end{defin}
The key observation, first noted by Sj\"orgen in \cite[Theorem 1]{sjogren1973adjoint} in the Euclidean setting, is that every distributional $\alpha$-subharmonic function is almost everywhere equal to a function which is $\alpha$-subharmonic in the sense of potential theory.
Note that in \cite[Theorem 1]{sjogren1973adjoint}, Sj\"orgen considers a wider class of elliptic differential operators. The drifted Laplace-Beltrami operator falls into that class.

More precisely, if $v\in L^1(\Omega)$ satisfies $\D_\alpha v \geq 0$ in the sense of distributions, then $v$ is equal almost everywhere to an $\alpha$-subharmonic function in the sense of potential theory. 
Naturally, if $v$ has some better regularity property, for example it is continuous, the equality holds everywhere. 
This fact holds true also in the Riemannian case, we sketch here the proof for clarity of exposition.

Recall that for every $\varphi \in C^\infty_c(\Omega)$ we have
\begin{equation*}
    \varphi(x)= - \int_\Omega G(x,y)\Delta_\alpha \varphi (y)\ \alpha^2(y)d\mu_y.
\end{equation*}
Furthermore, since $\D_\alpha v = d\nu^v$ is a positive Radon measure, we have
\begin{equation*}
    \int_\Omega v(x)\Delta_\alpha \varphi(x)\ \alpha^2(x)d\mu_x=\int_\Omega \varphi(x)\ d\nu^v_x
\end{equation*}
for every $\varphi \in C^\infty_c(\Omega)$.
The measure $d\nu^v$ is often referred as the $\D_\alpha$-Riesz measure of $v$.
By a direct computation we have 
\begin{align*}
    \int_\Omega v(x)\Delta_\alpha \varphi(x)\ \alpha^2(x)d\mu_x &=\int_\Omega \varphi(x)\ d\nu^v_x\\
    &= - \int_\Omega \int_\Omega G(x,y)\Delta_\alpha \varphi(y)\ \alpha^2(y) d\mu_y\ d\nu^v_x\\
    &=\int_\Omega - \left(\int_\Omega G(x,y) d\nu^v_x \right) \Delta_\alpha \varphi(y) \alpha^2(y)d\mu_y, 
\end{align*}
hence, 
\begin{equation*}
    \int_\Omega \left(v(y)+\int_\Omega G(x,y)\ d\nu^v_x \right) \Delta_\alpha \varphi (y) \alpha^2(y) d\mu_y=0,
\end{equation*}
for every $0\leq \varphi\in C^\infty_c(\Omega)$.
In other words, the function 
\begin{equation*}
    v + \int_\Omega G(x, \cdot) d\nu^v_x
\end{equation*}
is $\alpha$-harmonic in the sense of distributions. 
By \cite[Theorem 1]{sjogren1973adjoint} of Sj\"orgen we know that $\alpha$-harmonic functions are almost everywhere equal to a function which is $\alpha$-harmonic in the sense of potential theory. 
When the operator at hand is the Euclidean Laplacian, this result is usually referred as Weyl's lemma. 
We conclude that 
\begin{equation}
    \label{eq:representation for alpha subharmonic sjogren}
    v\overset{a.e.}{=}h-\int_\Omega G(x,\cdot) d \nu_x^v,
\end{equation}
where $h$ is $\alpha$-harmonic in a strong sense.
On the other hand, one can prove that the function 
\begin{equation}
    - G \ast d\nu^v = -\int_\Omega G(x,\cdot) d \nu_x^v
\end{equation}
is $\alpha$-subharmonic in the sense of potential theory which concludes the sketch of the proof. 
For this latter statement, we refer to Section 6 of \cite{bonfiglioli2013subharmonic}.

\subsection{Proof of Theorem \ref{Thm:MonotoneApproximationDriftedLaplacian}} 
In order to prove \Cref{Thm:MonotoneApproximationDriftedLaplacian}, we
adopt a strategy laid out by Bonfiglioli and Lanconelli in \cite{bonfiglioli2013subharmonic}, where they obtained some monotone approximation results for a wide class of second order elliptic operators on $\R^n$.
To do so, we begin by defining the following mean integral operators. 
If $v$ is an upper semicontinuous function on $\Omega$, $x \in \Omega$ and $r > 0$, we set 
\begin{equation}
    m_r(v)(x) \coloneqq \int_{\partial \mathcal{B}_r(x)}{v(y) | \nabla_y G(x,y)|  \alpha^2(y)\ d\sigma_y}. 
\end{equation}

In particular, if $v$ is an $\alpha$-subharmonic function in the sense of distributions we have the following results, which are an adaptation to the case or Riemannian manifolds of \cite{bonfiglioli2013subharmonic}. 
\begin{prop}
\label{prop:MrProperties}
Given a Riemannian manifold $(M,g)$ and $\Omega\Subset M$, if $v\in L^1(\Omega)$ is $\alpha$-subharmonic in the sense of distributions, then
\begin{enumerate}[label = (\roman*)]
    \item $v(x)\leq m_r(v)(x)$ for almost every $x \in \Omega$ and almost every $r>0$;
    \item let $0<s<r$ then $m_s(v)(x)\leq m_r(v)(x)$ almost everywhere in $\Omega$;
    \item for almost every $x\in \Omega$ we have $\lim_{r\to 0} m_r(v)(x)=v(x)$;
    \item for every $r>0$ $m_r(v)$ is $\alpha$-subharmonic in the sense of potential on $\Omega$. 
\end{enumerate}
\begin{proof} 
By the observation in the previous section, up to a choice of a good representative, we can assume that $v$ is $\alpha$-subharmonic in the sense of potential, cf. \Cref{def:alpha subharmonic potential}.
\begin{enumerate}[label = (\roman*), wide, labelindent=0pt]
    \item Fix $x_0\in \Omega$ and $r>0$, consider $\varphi\in C^0(\partial \mathcal{B}_r(x_0))$ such that $v\leq \varphi$ on $\partial \Br(x_0)$. 
    Let $h:\Br(x_0)\to \R$ be the solution of
    \begin{equation}
    \label{Eq:HHarmonicSolution}
    \begin{cases}
        \D_\alpha h=0 & \textnormal{in}\ \Br(x_0)  \\
             h=\varphi & \textnormal{on}\ \partial \Br(x_0) 
    \end{cases}.
    \end{equation}
    Since $v$ is $\alpha$-subharmonic in the sense of potential, then $v\leq h$ in $\Br(x_0)$.
    By Proposition \ref{Prop:RepresentationFormula} we have 
    \begin{equation}
    \label{Eq:DiseqHHarmonic}
        v(x_0)\leq h(x_0)=\int_{\pBr(x_0)}\varphi(y) |\nabla_y G(x_0,y)| d\sigma_y^\alpha
    \end{equation}
    where $d\sigma_y^\alpha=\alpha^2(y)\ d\sigma_y$.
    Since $v$ is upper semicontinuous on $\pBr(x_0)$, there exists a sequence $\{\varphi_i\}_i \subset C^0(\pBr(x_0))$ such that $\varphi_i(y)\searrow v(y)$ almost everywhere on $\pBr(x_0)$.
    Applying \eqref{Eq:DiseqHHarmonic} to each $\varphi_i$ we obtain by Dominated Convergence that
    \begin{equation*}
        v(x_0)\leq \int_{\pBr(x_0)}{v(y) | \nabla_y G(x_0,y)| d\sigma_y^\alpha}=m_r(v)(x_0).
    \end{equation*}
    
    \item Fix $0<s<r$, let $\varphi$ and $h$ be as in $(i)$ so that $v \le h$ on $\Br(x_0)$.
    By \Cref{Prop:RepresentationFormula} we have 
    \begin{equation*}
        m_s(v)(x_0) \leq \int_{\partial \mathcal{B}_s(x_0)}{h(y)|\nabla_y G(x_0,y)| d\sigma_y^\alpha}=h(x_0) = \int_{\pBr(x_0)}{\varphi(y) |\nabla_y G(x_0,y)| d\sigma_y^\alpha}.
    \end{equation*}
    Taking a monotone sequence of continuous functions on the boundary $\varphi_i\searrow u$ and proceeding as above we conclude
    \begin{align*}
        m_s(v)(x_0)\leq \int_{\pBr(x_0)}{\varphi_i(y) |\nabla_y G(x_0,y)| d\sigma_y^\alpha}\longrightarrow m_r(v)(x_0).
    \end{align*}
    
    \item This property is a consequence of the fact that $v$ is (almost everywhere) equal to an upper semicontinuous function. 
    Fix $x_0\in \Omega$ and $\varepsilon>0$ there exists a small enough neighborhood of $x_0$, $V(x_0)$, such that .
    \begin{equation*}
        v(y)<v(x_0)+\varepsilon
    \end{equation*}
    on $V(x_0)$. 
    Taking for $r>0$ small enough, we have
    \begin{equation*}
        m_r(v)(x_0) \leq v(x_0)+\varepsilon.
    \end{equation*}
    Recall that the function constant to $1$ is $\alpha$-harmonic on $\Omega$.
    By $(i)$, $v(x_0)\leq m_r(v)(x_0)$ hence
    \begin{equation*}
        m_r(v)(x_0)-\varepsilon\leq v(x_0) \leq m_r(v)(x_0).
    \end{equation*}
    Letting $\varepsilon$, and thus $r$ go to $0$, we obtain desired property.
    \item This last property is a consequence of the decomposition of $\alpha$-subharmonic functions observed in \eqref{eq:representation for alpha subharmonic sjogren}. 
    Integrating against $|\nabla G| \alpha^2$ both sides of \eqref{eq:representation for alpha subharmonic sjogren} we obtain 
    \begin{equation*}
        m_r(v)(x) = h(x) - m_r(G \ast d\nu^v)(x).
    \end{equation*}
    The desired property follows from the fact that the mean integral $- m_r(G \ast d\nu^v)$ is $\alpha$-subharmonic in the sense of potential. 
    For details we refer to Section 6 of \cite{bonfiglioli2013subharmonic}.
\end{enumerate}
\end{proof}
\end{prop}

The next step is to take a convolution of the mean integral functions $m_r(v)$ so to obtain smooth functions which produce the desired approximating sequence $\{v_k\}_k$.

\begin{proof}[Proof of Theorem \ref{Thm:MonotoneApproximationDriftedLaplacian}]
Let $\varphi\in C^1_c([0,1])$ be a non-negative function with unitary $L^1$-norm, we define
\begin{equation}
    \label{Def:Uk}
    v_k(x) \coloneqq k \int_0^{+\infty}{\varphi(ks)\ m_s(v)(x) ds}
\end{equation}
As shown in \cite{bonfiglioli2013subharmonic} the functions defined by \eqref{Def:Uk} are eventually smooth.

The monotonicity of the approximating sequence follows immediately from the monotonicity of $m_r(v)$ with respect to $r$.
Combining this with property $(i)$ of \Cref{prop:MrProperties} we obtain $(i)$. 
The proof of $(ii)$ is a consequence of $(iii)$ in \Cref{prop:MrProperties}. 
Both this proofs are straightforward computations, we refer to \cite[Theorem 7.1]{bonfiglioli2013subharmonic} for the details. 
The convergence in $L^1(\Omega)$ follows from $(i)$ and $(ii)$, using the fact that $|v_k| \leq \max\{|v|, |v_1|\}$ and Dominated Convergence. 
For the uniform estimate of $(iv)$, it is enough to observe that $1$ is an $\alpha$-harmonic function on $\Omega$ and $\varphi$ has unitary $L^1$ norm, hence, 
\begin{equation*}
    v_k(x) = k\int_0^{+\infty}{\varphi(ks)\ m_s(v)(x)\ ds} \leq \esup_{\Omega} v k \int_0^{+\infty}{\varphi(ks)\ m_s(1)(x)\ ds} = \esup_{\Omega}v.
\end{equation*}
This concludes the proof of \Cref{Thm:MonotoneApproximationDriftedLaplacian}. 
\end{proof}

\begin{remark}
\label{rmk:observation on ess sup}
Note that in the last estimate, one actually has
\begin{equation*}
    \esup_\Omega v_k \leq \underset{\mathcal{B}_{1/k}(x)}{\esup} v \le \esup_\Omega v. 
\end{equation*}
This observation will be crucial later on. 
\end{remark}

\subsection{Proof of \texorpdfstring{\Cref{thm:monotone approximation L intro}}{Theorem D}}
Finally, we desume the proof of \Cref{thm:monotone approximation L intro} from \Cref{Thm:MonotoneApproximationDriftedLaplacian}.
If $\{v_k\}_k$ is the approximating sequence for the function $v=\frac{u}{\alpha}$, we define $u_k \coloneqq \alpha v_k$. 
By \Cref{Lem:EquivalenceOfPositivity_Formula1},  $\{u_k\}_k$ is an approximating sequence for $u$ as it satisfies $(i) - (iii)$ of Theorem \ref{thm:monotone approximation L intro}. 
The proof is trivial and is therefore omitted. 
A little more effort is required to show that if $\sup_\Omega v_k \leq \esup_\Omega v$, then $\sup_\Omega u_k \leq 2 \esup_\Omega u$ for $k$ large enough.

To this end, fix $x\in \Omega$. 
As noted in \Cref{rmk:observation on ess sup} we have
\begin{equation*}
    u_k(x)=\alpha(x) v_k(x)\leq \alpha(x) \esup_{\mathcal{B}_{1/k}(x)}v\leq \frac{\alpha(x)}{\inf_{\mathcal{B}_{1/k}}\alpha} \esup_\Omega u. 
\end{equation*}
Furthermore, for every $y\in \mathcal{B}_{1/k}(x)$ we estimate
\begin{equation}
    \label{eq:alpha(y)/alpha(x)}
    \frac{\alpha (x)}{\alpha (y)}\leq \frac{|\alpha(x)-\alpha(y)|}{\alpha(y)}+1\leq \frac{r_k(x) \sup_{\Omega}|\nabla \alpha|}{\inf_\Omega \alpha}+1
\end{equation}
where $r_k(x)=\sup\{d(x,z) : z\in \mathcal{B}_{1/k}(x)\}$.
Next, we show that the function $r_k(x)$ can be uniformly bounded so that \eqref{eq:alpha(y)/alpha(x)} is bounded above by $2$.

\begin{lem}
\label{lem:rk(x)}
There exists some $k_0\in \mathbb{N}$ such that
\begin{equation*}
    r_k(x) \leq \frac{\inf_\Omega \alpha}{\sup_\Omega |\nabla \alpha|} =:c \qquad \forall x \in \Omega, \quad \forall k \geq k_0. 
\end{equation*}
\begin{proof} 
Suppose by contradiction that there exists a sequence of points $\{x_k\}_k\subset \Omega$ such that $r_k(x_k)>c$ for every $k\in \mathbb{N}$. 
By definition of $r_k(x_k)$, there exists a sequence of points $\{y_k\}_k \subset \mathcal{B}_{1/k}(x_k)$ such that $d(y_k,x_k)>c$. 
Since $\Omega$ is relatively compact, up to a subsequence, we can assume that $x_k\to x_\infty \in \overline{\Omega}$ and $y_k\to y_\infty \in \overline{\Omega}$. Since $y_k \in \mathcal{B}_{1/k}(x_k)$ we have 
\begin{equation}
    \label{eq:G to infty}
    G(x_k,y_k)>k \to +\infty. 
\end{equation}
Note also that the Green function $G$ is smooth and hence continuous on $\Omega \times \Omega \setminus \{x = y\}$.
Note that since $d(x_k, y_k) > c$, then  $d(x_\infty, y_\infty) \ge c$, in particular we deduce that $x_\infty \not\in \partial \Omega$ because the Green function $G$ vanishes on the boundary of $\Omega$.
If $x_\infty \in \Omega$ is not on the boundary, fix $\overline{k} \in \N$. 
By \eqref{eq:G to infty} and continuity of the Green function we have $G(y_\infty,x_\infty) > \overline{k} $ which implies that $y_\infty \in \mathcal{B}_{1/\overline{k} }(x_\infty)$.
In particular we have $d(x_\infty, y_\infty) \leq r_{\overline{k}}(x_\infty) \to 0$, which is a contradiction since $d(x_\infty, y_\infty) \ge c$. 
Indeed, for every $x \in \Omega$,
\begin{equation*}
    \lim_{k \to + \infty} r_k(x) = 0. 
\end{equation*}
Clearly, $r_k(x)$ is a monotone decreasing sequence in $k$. 
If its limit is some $r_0 \neq 0$ this implies that $r_k(x) \geq r_0$ for all $k$. 
In particular the geodesic ball $B_{r_0}(x)$ is contained in $\mathcal{B}_{1/k}(x)$ for all $k \in \N$. 
This, however, is a contradiction since 
\begin{equation*}
    \bigcap_{k=1}^{\infty} \mathcal{B}_{1/k}(x)=\{x\}. 
\end{equation*}
\end{proof}
\end{lem}
Thanks to \Cref{lem:rk(x)}, up to taking $k$ large enough, we have 
\begin{equation*}
    \alpha(x)\leq 2\alpha (y) \qquad \forall x \in \Omega \text{ and } \forall y \in \mathcal{B}_{1/k}(x),
\end{equation*}
hence,
\begin{equation*}
    u_k(x)\leq \frac{\alpha(x)}{\inf_{\mathcal{B}_{1/k}}\alpha} \esup_\Omega u \leq 2 \esup_\Omega u \qquad \forall x \in \Omega.
\end{equation*}
This concludes the proof of \Cref{thm:monotone approximation L intro}.

\subsection{Remarks on the global case}
\label{sec:global case}
A careful analysis of above proofs shows that the monotone approximation results can be obtained globally on the whole manifold $M$ as long as there exists a minimal positive Green function for the operator $\D_\alpha$ and the super level sets $\mathcal{B}_{r}(x)$ are compact. 
Not all Riemannian manifolds, however, satisfy these conditions. 
We recall the following

\begin{defin}
A Riemannian manifold $(M,g)$ is said to be \textit{$\alpha$-non-parabolic} if there exists a minimal positive Green function $G$ for the operator $\D_\alpha$. Moreover, if this Green function satisfies
\begin{equation}
    \label{Def:StronglyNonParabolic}
    \lim_{y\to \infty}G(x,y)=0,
\end{equation}
the manifold $M$ is said to be \textit{strongly $\alpha$-non-parabolic}.
\end{defin}

Note that compact Riemannian manifold are always $\alpha$-parabolic thus we focus on the complete, non-compact case. 
It is also known that if $(M, g)$ is a geodesically complete, $\alpha$-non-parabolic manifold, then 
\begin{equation}
    \label{Eq:VolumeGrowth}
    \int_1^\infty \frac{t}{\vol_\alpha (B_t(x))} dt<\infty
\end{equation}
where $\vol_\alpha(B_t(p))$ is the volume of the geodesic ball of radius $t$ and center $x$ with respect to the measure $\alpha^2 d\mu_g$.
See for instance Theorem 9.7 of \cite{Greg2006}.
Furthermore, if we assume a non-negative $m$-Bakry-\'Emery Ricci tensor $\Ric^m_f \coloneqq \Ric + \Hess(f) - \frac{1}{m} df \otimes df \ge 0$ with $f = -2\log{\alpha}$, it is possible to prove some Li-Yau type estimates for the heat kernel, see Theorems 5.6 and 5.8 in \cite{CL2015}. 
Integrating in time these estimates we obtain the following bounds for the Green function 
\begin{equation*}
    C^{-1} \int_{d(x,y)}^\infty \frac{t}{\vol_\alpha (B_t(x))} dt \leq G(x,y) \leq C \int_{d(x,y)}^\infty \frac{t}{\vol_\alpha (B_t(x))} dt. 
\end{equation*}
In particular if \eqref{Eq:VolumeGrowth} holds true and $\Ric^m_f \ge 0$, the previous estimate implies that the manifold at hand is strongly $\alpha$-non parabolic. 
It would be interesting to investigate which geometric conditions on the manifold $(M, g)$ imply the existence of a function $\alpha$ such that \eqref{Eq:VolumeGrowth} and $\Ric^m_f \ge 0$ hold true. 

\section{A counterexample to the \texorpdfstring{$L^1$}{L 1}-positivity preserving property}
\label{sec:counterexample L1}
This section is devoted to the proof of \Cref{thm:L1 ppp intro}.
Fix $\varepsilon>0$ and consider the 2-dimensional model manifold $M =\R_{+}\times_{\sigma}\mathbb{S}^1$, that is $\R_{+}\times\mathbb{S}^1$ with the metric $g = dt^2 + \sigma^2(t) d\theta^2$. 
Here $d\theta^2$ is the standard round metric on $\mathbb{S}^1$ and $\sigma = \sigma_\varepsilon$ is a $C^\infty((0,+\infty))$ function satisfying
\begin{equation*}
    \sigma(t) =
    \begin{cases}
    j(t) &t> t_\varepsilon \\
    t &t< \frac{1}{4}
    \end{cases}.
\end{equation*}
Here $t_\varepsilon = \left(2(1+\varepsilon)\varepsilon\right)^{-1/2\varepsilon}$ and the function $j$ is defined as
\begin{equation*}
    j(t) = \frac{e^{-t^{2+2\varepsilon}}}{t^{1+\varepsilon}}. 
\end{equation*}
By a direct computation we have 
\begin{align*}
    j'(t) &=-(1+\varepsilon) e^{-t^{2+2\varepsilon}} \left(2t^\varepsilon+\frac{1}{t^{2+\varepsilon}} \right)\\
    j''(t) &= (1+\varepsilon) e^{-t^{2+2\varepsilon}} \left[2t^{\varepsilon-1} + 4(1+\varepsilon)t^{1+3\varepsilon} + (2+\varepsilon)\frac{1}{t^{3+\varepsilon}}\right]. 
\end{align*}
As a result, outside of a compact set we have the following asymptotic estimate for the Gaussian curvature:
\begin{align*}
   K(t, \theta)&=-\frac{j''(t)}{j(t)}g\\
    & = -(1+\varepsilon) \left[2 t^{2\varepsilon} + 4(1+\varepsilon) t^{2+4\varepsilon} + (2+\varepsilon)\frac{1}{t^2}\right] g \\
    &\sim -4 (1+\varepsilon)^2 t^{2+4\varepsilon}g
\end{align*}
as $t \to + \infty$. 
Next we define the function  $U(t,\theta)=u(t)=(e^{t^{2+2\varepsilon}}-e^{t_\varepsilon^{2+2\varepsilon}})_+$ and prove that it satisfies 
\begin{equation*}
    \Delta U\geq U
\end{equation*}
in the sense of distributions. 
If $t > t_\varepsilon$, by direct computation we have 
\begin{align*}
    u'(t)&= 2(1+\varepsilon)t^{1+2\varepsilon} e^{t^{2+2\varepsilon}}\\
    u''(t)&=2(1+\varepsilon) e^{t^{2+2\varepsilon}} \left[2(1+\varepsilon) t^{2+4\varepsilon}+(1+2\varepsilon)t^{2\varepsilon} \right]
\end{align*}
thus
\begin{equation*}
    \Delta U-U=u''(t)+\frac{j'(t)}{j(t)} u'(t)-u(t)=e^{t^{2+2\varepsilon}}\left[2(1+\varepsilon)\varepsilon t^{2\varepsilon}-1\right]+e^{t_\varepsilon^{2+2\varepsilon}}\geq 0. 
\end{equation*}
On the other hand, if $t< t_\varepsilon$ the function $U$ is identically zero, so that $\Delta U - U \geq 0$ also for $t\in (0, t_\varepsilon)$.
To see that $\Delta U\geq U$ in the sense of distributions on the whole manifold we take $0\leq \varphi \in C^\infty_c(M)$ and set  $\overline{M} \coloneqq M \setminus B_{t_\varepsilon}(0)$. 
Then we compute
\begin{align*}
    \int_{M} U(\Delta \varphi -\varphi)&= \int_{\overline{M}} U(\Delta \varphi-\varphi)\\
    &=-\int_{\overline{M}} g(\nabla \varphi, \nabla U) + \int_{\partial \overline{M}} U \frac{\partial \varphi}{\partial \nu}- \int_{\overline{M}} U \varphi\\
    &=-\int_{\overline{M}} g(\nabla \varphi, \nabla U) - \int_{\overline{M}} U \varphi\\
    &=\int_{\overline{M}}\Delta U \varphi - \int_{\partial \overline{M}} \frac{\partial U}{\partial \nu} \varphi - \int_{\overline{M}} U \varphi\\
    &=\int_{\overline{M}}\Delta U \varphi + \int_{\partial B_{t_\varepsilon}(0)} \frac{\partial U}{\partial t} \varphi - \int_{\overline{M}} U \varphi\\
    &=\int_{\overline{M}}(\Delta U-U) \varphi + \int_{\partial B_{t_\varepsilon}(0)} u' \varphi \geq 0.
\end{align*}
On the other hand we have:
\begin{equation*}
    \int_M |U| dV_g = \omega_m\int_0^{+\infty} u(t) j(t) dt = \int_{t_\varepsilon}^{+\infty} \frac{1}{t^{1+\varepsilon}} dt < + \infty. 
\end{equation*}
In conclusion, if we set $V = -U$ we have $V \in L^1(M)$ and $(-\D + 1)V \ge 0$ but $V \le 0$, which contradicts the validity of the $L^1$-positivity preserving property on $M$. 
\bibliography{references}
\bibliographystyle{acm}
\end{document}